\documentclass[article,oneside,oldfontcommands,a4paper,11pt]{memoir}

\usepackage[T1]{fontenc}

\usepackage[numbers,sort&compress]{natbib}
\usepackage[hypertexnames=false,colorlinks,linkcolor={blue},citecolor={blue}]{hyperref}

\usepackage{amssymb,amsmath,mathrsfs,stmaryrd,eucal}
\usepackage{amsthm,empheq}

\usepackage{xcolor}
\usepackage{graphicx,float,caption}
\usepackage{epsfig,subcaption}

\usepackage{enumitem}
\setlist{
leftmargin=2em, 
topsep = 4pt, 
itemsep = 1pt, 
}

\usepackage{gmacro}
\title{An existence theorem for single leader multi-follower games with direct preference maps}

\author[1]{Yutthakan Chummongkhon}
\author[1,2]{Poom Kumam}
\author[1,2]{Parin Chaipunya\thanks{Corresponding author}$^{,}$}

\affil[1]{
  {Department of Mathematics, Faculty of Science, King Mongkut's University of Technology Thonburi},
  {126 Pracha Uthit Rd.},
  {Bang Mod, Thung Khru, Bangkok},
  {10140}, 
  {Thailand.}
}
\affil[2]{
  {Center of Excellence in Theoretical and Computational Science (TaCS-CoE), King Mongkut's University of Technology Thonburi},
  {126 Pracha Uthit Rd.},
  {Bang Mod, Thung Khru, Bangkok},
  {10140}, 
  {Thailand.}

  Email: yuttagan.24@mail.kmutt.ac.th, parin.cha@kmutt.ac.th and  poom.kum@kmutt.ac.th;
  \vspace{.25cm}
}

\date{}

\begin{document}
\maketitle \vspace{-1.8cm}
\thispagestyle{draftfirstpage}

\begin{abstract}
  This paper concerns with an existence of a solution for a single leader multi-follower game (SLMFG), where the followers jointly solve an abstract economy problem.
  Recall that an abstract economy problem is an extension of a generalized Nash equilibrium problem (GNEP) in the sense that the preference of each player can be described without numerical criteria.
  The results of this paper therefore extend the known literature concerning an existence of a solution of a SLMFG.
  Due to the lack of criterion functions in the lower-level, the technique we used is quite different from those that deal with GNEPs.
  In particular, we argue that the follower's abstract economy profiles, {\itshape i.e.,} constraint-preference couples, belong to a particular metric space where the response map is proved to be upper semicontinuous.

\end{abstract}


\section{Introduction}

A bilevel programming, or bilevel optimization problem, refers to a class of optimization problems that involve two players, one is classified as the \emph{leader} and the other one as the \emph{follower}.
When the leader makes her decision, the follower observes and makes an \emph{optimal response} by solving the \emph{lower-level} optimization problem, parametrized by the leader's action.
The goal of the leader is then to make a decision so that, altogeter with the follower's optimal response, gives her the \emph{best} objective value.
When the single follower is replaced with multiple followers, they reponse to the leader's action by solving a parametrized equilibrium-type problem.
This approach is broadly referred to as a \emph{single leader multi-follower game} (or briefly, \emph{SLMFG}).
These classes of bilevel problems have posed several interesting and important questions, theoretically and practically, in the literature. 
For instance, the issues regarding the KKT reformulation (or MPCC reformulation) were discussed in~\cite{zbMATH06008674, zbMATH07063784, arXiv:2503.14962}.
The existence results for several different settings of a SLMFG were studied in~\cite{zbMATH07289896, zbMATH07462111, zbMATH07794717, zbMATH07429239, zbMATH07950122}.

In this paper, we propose an existence theorem for a single leader multi-follower game (SLMFG) in which the followers' preferences are directly given by the preference maps without the need of objective functions.
Thus, the lower-level problem here is presented as an abstract economy, \`a la~\citet{zbMATH03488939}, parametrized by the leader's action.
Let us state precisely the problem of our study for this paper.
Consider the~$\numf$ followers that are represented by a set~$\Follower := \set{1,\dots,\numf}$.
The single leader, denoted by~$\leader$, owns an action set~$\lControl$ and each follower~$\follower \in \Follower$ owns an action set~$\fControl_{\follower}$.
We write~$\fControl := \prod_{\follower \in \Follower} \fControl_{\follower}$ and for each~$\follower \in \Follower$ we write~$\fControl_{-\follower} := \prod_{\gollower \in \Follower \setminus \set{\follower}} \fControl_{\gollower}$ whose generic element is denoted by~$\fcontrol_{-\follower}$.
It is typical to represent~$\fControl$ as~$\fControl_{\follower} \times \fControl_{-\follower}$ and to decompose~$\fcontrol \in \fControl$ into~$(\fcontrol_{\follower},\fcontrol_{-\follower}) \in \fControl_{\follower} \times \fControl_{-\follower}$, and also to extract~$\fcontrol_{-\follower}$ from~$\fcontrol$ the \emph{other followers' component}.
The leader is equipped with an objective function~$\lcriterion : \lControl \times \fControl \to \R$, while each follower~$\follower \in \Follower$ is equipped with a constraint map~$\Cons_{\follower} : \lControl \times \fControl_{-\follower} \multimap \fControl_{\follower}$ and a preference map~$\Pref_{\follower} : \lControl \times \fControl \multimap \fControl_{\follower}$.
For each~$\lcontrol \in \lControl$ and~$\follower \in \Follower$, it is customary to write~$\Cons_{\follower}^{\lcontrol}(\cdot) := \Cons_{\follower}(\lcontrol,\cdot)$ and~$\Pref_{\follower}^{\lcontrol}(\cdot) := \Pref_{\follower}(\lcontrol,\cdot)$.
The SLMFG that we focus on in this paper has the following form:
\begin{subequations}\label{eqn:SLMFG}
  \begin{empheq}[left=\empheqlbrace]{align}
    \min_{\lcontrol,\fcontrol} \quad
        & \lcriterion(\lcontrol,\fcontrol) \\
      \text{s.t.}  \quad
        & \lcontrol \in \lCons \subset \lControl  \\
        & \forall \follower \in \Follower: \label{eqn:NE_AE1}\\
        & \ \lfloor \  \fcontrol_{\follower} \in \Cons_{\follower}^{\lcontrol}(\fcontrol_{-\follower}), \quad \Cons_{\follower}^{\lcontrol}(\fcontrol_{-\follower}) \cap \Pref_{\follower}^{\lcontrol}(\fcontrol) = \emptyset. \label{eqn:NE_AE2}
  \end{empheq}
\end{subequations}
In particular, the constraints~\eqref{eqn:NE_AE1}--\eqref{eqn:NE_AE2} describe that the followers respond to the leader's signal~$\lcontrol$ with a \emph{Nash equilibrium} of a~\emph{$\lcontrol$-parametrized abstract economy} given by the pair of maps~$(\Cons_{\follower}^{\lcontrol},\Pref_{\follower}^{\lcontrol})_{\follower \in \Follower}$.
For convenience, at each~$\lcontrol \in \lControl$ we write
\begin{align*}
  \NE(\lcontrol) := \set{\fcontrol = (\fcontrol_{\follower})_{\follower \in \Follower} \in \fControl \mid \forall \follower \in \Follower:\ \fcontrol_{\follower} \in \Cons_{\follower}^{\lcontrol}(\fcontrol_{-\follower}), \quad \Cons_{\follower}^{\lcontrol}(\fcontrol_{-\follower}) \cap \Pref_{\follower}^{\lcontrol}(\fcontrol) = \emptyset}
\end{align*}
to denote the set of all such Nash equilibrium points.

The SLMFG~\eqref{eqn:SLMFG} indeed covers many particular cases of Nash equilibrium-like problems.
For instance, if one fix~$\varepsilon \geq 0$ and set
\begin{align*}
  & \Pref_{\follower}^{\lcontrol}(\fcontrol_{\follower},\fcontrol_{-\follower}) := \set{\fycontrol_{\follower} \in \fControl_{\follower} \mid \fcriterion_{\follower}(\lcontrol,\fycontrol_{\follower},\fcontrol_{-\follower}) < \fcriterion_{\follower}(\lcontrol,\fcontrol_{\follower},\fcontrol_{-\follower}) - \varepsilon, \quad (\forall \fycontrol \in \fControl_{\follower})}, \\
  & \fcriterion_{\follower} : \lControl \times \fControl \to \R,
\end{align*}
then~\eqref{eqn:SLMFG} reduces to
\begin{subequations}
  \begin{empheq}[left=\empheqlbrace]{align}
    \min_{\lcontrol,\fcontrol} \quad
        & \lcriterion(\lcontrol,\fcontrol) \\
      \text{s.t.}  \quad
        & \lcontrol \in \lCons \subset \lControl  \\
        & \forall \follower \in \Follower: \\
        & \ \lfloor \ \fcontrol_{\follower} \in \varepsilon-\argmin_{\fycontrol_{\follower}} \set{\fcriterion_{\follower}(\lcontrol,\fycontrol_{\follower},\fcontrol_{-\follower}) \mid \fycontrol_{\follower} \in \Cons_{\follower}^{\lcontrol}(\fcontrol_{-\follower})},
  \end{empheq}
\end{subequations}
which is a single leader multi-follower game where the followers solves for an~$\varepsilon$-Nash equilibrium of a~$\lcontrol$-parametrized (generalized) Nash equilibrium problem.
When~$\varepsilon = 0$, this is a classical single leader multi-follower game.

Likewise, one could also retrieve the case where the followers play a set-valued Nash equilibrium problem by setting
\begin{subequations}
  \begin{empheq}[left=\empheqlbrace]{align*}
    & \Pref_{\follower}^{\lcontrol}(\fcontrol_{\follower},\fcontrol_{-\follower}) := \set{\fycontrol_{\follower} \in \fControl_{\follower} \mid (\fycontrol_{\follower},\fcontrol_{-\follower}) \succeq_{\follower} (\fcontrol_{\follower},\fcontrol_{-\follower}) }, \\
    & \succeq_{\follower} \ \subset \fControl \times \fControl.
  \end{empheq}
\end{subequations}
Then~$\Cons_{\follower}^{\lcontrol}(\fcontrol_{-\follower}) \cap \Pref_{\follower}^{\lcontrol}(\fcontrol_{\follower},\fcontrol_{-\follower}) = \emptyset$ describes that there is no feasible point~$\fycontrol_{\follower} \in \Cons_{\follower}^{\lcontrol}(\fcontrol_{-\follower})$ in which~$(\fycontrol_{\follower},\fcontrol_{-\follower})$ is not preferred over~$(\fcontrol_{\follower},\fcontrol_{-\follower})$.
In the language of vector optimization, this says that~$(\fcontrol_{\follower},\fcontrol_{-\follower})$ is \emph{weakly efficient}.

The aim of this paper is to develop an existence theorem for a SLMFG of the form~\eqref{eqn:SLMFG} where the followers' problem is an abstract economy problem parametrized by the leader's decision.
We take the classical approach by showing that the solution map (or the response map) of the lower-level problem has closed graph.
For this, we characterize the solution map with a parametric minimizer set of its associated gap function.
We introduce a regularity condition and show that, under the uniform metric, the solution map is upper semicontinuous.

\section{Preliminaries}

In this section, we recall basic notions in set-valued analysis which includes distance between points and sets, convergence of sets, and continuity of set-valued maps.

Suppose that~$\Domain$ and~$\CoDomain$ are subsets of (possibly different) Banach spaces. A set-valued map~$\SetMap : \Domain \multimap \CoDomain$ is said to be
\begin{itemize}[label=$\circ$, leftmargin=*]
  \item \emph{closed} if its graph~$\gr(\SetMap) := \set{(\xpoint,\ypoint) \in \Domain \times \CoDomain \mid \ypoint \in \SetMap(\xpoint)}$ is closed in~$\Domain \times \CoDomain$,
  \item \emph{compact} if~$\gr(\SetMap)$ is compact in~$\Domain \times \CoDomain$,
  \item \emph{closed-valued} if~$\SetMap(\xpoint)$ is closed for all~$\xpoint \in \Domain$,
  \item \emph{upper semicontinuous} on~$\SubDomain \subset \Domain$ if for every~$\xpoint \in \SubDomain$ and each open set~$\ONhood$ with~$\SetMap(\xpoint) \subset \ONhood$, there exists a neighborhood~$\VNhood$ of~$\xpoint$ such that~$\SetMap(\xpoint') \subset \ONhood$ for all~$\xpoint' \in \VNhood$,
  \item \emph{lower semicontinuous} on~$\SubDomain \subset \Domain$ if for every~$\xpoint \in \SubDomain$ and each open set~$\ONhood$ with~$\SetMap(\xpoint) \cap \ONhood \neq \emptyset$, there exists a neighborhood~$\VNhood$ of~$\xpoint$ such that~$\SetMap(\xpoint') \cap \ONhood \neq \emptyset$ for all~$\xpoint' \in \VNhood$,
  \item \emph{continuous} on~$\SubDomain \subset \Domain$ if it is both upper and lower semicontinuous on~$\SubDomain$.
\end{itemize}
The following results (see {\itshape e.g.} \cite{zbMATH00045282}) regarding~$\SetMap : \Domain \multimap \CoDomain$ are widely known in set-valued analysis:
\begin{enumerate}[label=(\roman*), leftmargin=*]
  \item If~$\SetMap$ is closed-valued and upper semicontinuous, then it is closed.
  \item If~$\SetMap$ is closed and~$\CoDomain$ is compact, then it is upper semicontinuous.
  \item If~$\CoDomain$ is compact, then~$\SetMap$ is upper semicontinuous at~$\xpoint \in \Domain$ if and only if for any sequences~$(\xpoint^{n})$ in~$\Domain$ and~$(\ypoint^{n})$ in~$\CoDomain$ with~$\xpoint^{n} \to \xpoint$ and~$\ypoint^{n} \in \SetMap(\xpoint^{n})$ for each~$n \in \N$, it holds that~$\ypoint \in \SetMap(\xpoint)$.
  \item If~$\CoDomain$ is compact, then~$\SetMap$ is lower semicontinuous at~$\xpoint \in \Domain$ if and only if for any sequence~$(\xpoint^{n})$ in~$\Domain$ with~$\xpoint^{n} \to \xpoint$ and any~$\ypoint \in \SetMap(\xpoint)$, there exists~$\ypoint^{n} \in \SetMap(\xpoint^{n})$ (for each~$n \in \N$) such that~$\ypoint^{n} \to \ypoint$.
\end{enumerate}

Let us denote with~$\Compact(\Domain)$ the family of compact subsets of~$\Domain$.
Recall that the distance on~$\Domain$ induces a metric on~$\Compact(\Domain)$ that preserves the completeness.
For any~$\xpoint \in \Domain$ and~$\DSet \in \Compact(\Domain)$, we define the \emph{displacement from~$\DSet$} by
\begin{align*}
  d(\xpoint,\DSet) := \inf_{\ypoint \in \DSet} d(\xpoint,\ypoint).
\end{align*}
Observe that~$\xpoint \in \DSet \iff d(\xpoint,\DSet) = 0$.
For any~$\CSet,\DSet \in \Compact(\Domain)$, the \emph{excess} of~$\CSet$ from~$\DSet$ is defined by
\begin{align*}
  e(\CSet,\DSet) := \sup_{\xpoint \in \CSet} \inf_{\ypoint \in \DSet} d(\xpoint,\ypoint).
\end{align*}
It is important to notice that~$e(\CSet,\DSet) \neq e(\DSet,\CSet)$ in general.
Finally, the Hausdorff distance between~$\CSet,\DSet \in \Compact(\Domain)$ is given by
\begin{align*}
  h(\CSet,\DSet) := \max\set{e(\CSet,\DSet), e(\DSet,\CSet)}.
\end{align*}
For any~$\DSet \subset \Domain$ and~$\varepsilon > 0$, we adopt the notation
\begin{align*}
  \Nhood_{\varepsilon}[\DSet] := \set{\xpoint \in \Domain \mid d(\xpoint,\DSet) < \varepsilon}
\end{align*}
to denote an open ball about a set~$\DSet$.
When~$\CSet = \set{\xpoint}$ is singleton, we prefer to write~$\Nhood_{\varepsilon}[\xpoint]$ instead of~$\Nhood_{\varepsilon}[\set{\xpoint}]$.
With this, one may also notice an equivalent formulation of the Hausdorff distance, which reads
\begin{align*}
  h(\CSet,\DSet) = \inf\set{\eta > 0 \mid \Nhood_{\eta}[\CSet] \supset \DSet, \quad \Nhood_{\eta}[\DSet] \supset \CSet}.
\end{align*}
It is well-known that~$(\Compact(\Domain),h)$ is a metric space which is complete if and only if~$\Domain$ is complete.
Exploiting the Hausdorff metric, we have a further characterization of continuity of a set-valued map.
That is, if~$\CoDomain$ is compact, then~$\SetMap : \Domain \multimap \CoDomain$ is continuous at~$\xpoint \in \Domain$ if and only if~$\lim_{n \to \infty} h(\SetMap(\xpoint^{n}),\SetMap(\xpoint)) = 0$ for any sequence~$(\xpoint^{n})$ in~$\Domain$ with~$\xpoint^{n} \to \xpoint$.

Next, consider a sequence~$(\CSet^{n})$ of subsets of~$\Domain$.
Its \emph{outer} and \emph{inner limits} could be defined, respectively, by
\begin{align*}
  \Limsup_{n \to \infty} \CSet^{n} := \Set{\xpoint \in \Domain \left|
    \begin{array}{c}
      \text{$\xpoint$ is an accumulation point of some sequence~$(\xpoint^{n})$} \\
      \text{such that~$\xpoint^{n} \in \CSet^{n}$ for all~$n \in \N$.}
    \end{array}\right.}
\end{align*}
and
\begin{align*}
  \Liminf_{n \to \infty} \CSet^{n} := \Set{\xpoint \in \Domain \left|
    \begin{array}{c}
      \text{$\xpoint$ is the limit of some convergent sequence~$(\xpoint^{n})$} \\
      \text{such that~$\xpoint^{n} \in \CSet^{n}$ for all~$n \in \N$.}
    \end{array}\right.}.
\end{align*}
We say that~$\CSet \subset \Domain$ is the \emph{Painlev\'e-Kuratowski limit} of~$(\CSet^{n})$, or symbolically that~$\CSet^{n} \overset{\rm{PK}}{\to} \CSet$, if
\begin{align*}
  \CSet = \Limsup_{n \to \infty} \CSet^{n} = \Liminf_{n \to \infty} \CSet^{n}.
\end{align*}
If~$\CSet$ and all~$\CSet^{n}$'s are compact, then~$\CSet^{n} \overset{\rm{PK}}{\to} \CSet$ if and only if~$\lim_{n \to \infty} h(\CSet^{n},\CSet) = 0$.
That is, the Painlev\'e-Kuratowski and the Hausdorff limits coincide for compact sequences.

\section{Stability analysis for abstract economy problems}

This section is dedicated to the stability analysis of the lower-level abstract economy.
The idea is that a decision~$\lcontrol$ of the leader \emph{selects} a tuple~$(\Cons_{\follower}^{\lcontrol},\Pref_{\follower}^{\lcontrol})_{\follower \in \Follower}$ representing an \emph{abstract economy profile}, from some particular class~$\Game$, that defines the followers's abstract economy problem.

\subsection{An existence result}

For the moment, we neglect the leader's ingredient and focus on the solvability of an abstract economy problem.
\begin{theorem}\label{thm:existence}
Consider a finite set~$\Follower$ of players.
  For each~$\follower \in \Follower$, let~$\fCONTROL_{\follower}$ be a Banach space and~$\fControl_{\follower} \subset \fCONTROL_{\follower}$ be nonempty, closed and convex.
  Suppose, for each~$\follower \in \Follower$, that the following assumptions hold
  \begin{enumerate}[label=(A\arabic*), leftmargin=*]
    \item\label{asmp:A1} the maps~$\Cons_{\nu} : \fControl_{-\follower} \multimap \fControl_{\follower}$ and~$\Pref_{\nu} : \fControl \multimap \fControl_{\follower}$ are upper semicontinuous with closed convex values,
    \item\label{asmp:A2} the map~$\Cons_{\nu}$ has nonempty values,
    \item\label{asmp:A3} the set~$\Improvement_{\follower} := \set{\fcontrol \in \fControl \mid \Cons_{\follower}(\fcontrol_{-\follower}) \cap \Pref_{\follower}(\fcontrol) = \emptyset}$ is open,
    \item\label{asmp:A4} for each~$\fcontrol \in \fControl$,~$\fcontrol_{\follower} \in \Cons_{\follower}(\fcontrol_{-\follower})$ implies~$\fcontrol_{\follower} \not\in \Pref_{\follower}(\fcontrol)$. 
  \end{enumerate}
  Then there is a point~$\bar{\fcontrol} \in \fControl$ such that
  \begin{equation} \label{eqn:Abstract_Econ}
    \bar{\fcontrol}_{\follower} \in \Cons_{\follower}(\bar{\fcontrol}_{-\follower}), \quad \Cons_{\follower}(\bar{\fcontrol}_{-\follower}) \cap \Pref_{\follower}(\bar{\fcontrol}) = \emptyset
  \end{equation}
  for all~$\follower \in \Follower$.
\end{theorem}
\begin{proof}
  Let~$\follower \in \Follower$. Define~$\Gmap_{\follower} : \fControl \multimap \fControl_{\follower}$ by
  \begin{align}\label{eqn:Gmap}
    \Gmap_{\follower}(\fcontrol) := 
      \begin{cases}
        \Cons_{\follower}(\fcontrol_{-\follower}) \cap \Pref_{\follower}(\fcontrol) &\text{if~$\fcontrol \in \Improvement_{\follower}$}, \\
        \Cons_{\follower}(\fcontrol_{-\follower}) &\text{otherwise.}
      \end{cases}
  \end{align}
  Let~$\Gmap := \prod_{\follower \in \Follower} \Gmap_{\follower}$.
  It is clear that if~$\fcontrol \in \fControl$ satisfies~\eqref{eqn:Abstract_Econ} for all~$\follower \in \Follower$, then~$\fcontrol \in \Gmap(\fcontrol)$.
  The converse holds with the use of~\ref{asmp:A4}. This means~\eqref{eqn:Abstract_Econ} is characterized with~$\fcontrol \in \fControl$ being a fixed point of the map~$\Gmap$.
  From~\cite[Theorem 3.1.8]{zbMATH00045282},~$\Gmap$ is upper semicontinuous and compact on~$\fControl$ with nonempty closed convex values.
  The existence of such a fixed point then follows from~\cite[Theorem 1 and Lemma 3]{Fan1952}
\end{proof}


\subsection{A stability result}

In this subsection, we consider a class of \emph{abstract economy profiles}, which are described by tuples~$(\Cons_{\follower},\Pref_{\follower})_{\follower \in \Follower}$ of constraint-preference maps, where the continuity in the perturbation implies the continuity of the corresponding Nash equilibrium points.

Always consider~$\Follower$ a finite set of players and for each~$\follower \in \Follower$,~$\fControl_{\follower}$ is a closed convex subset of a Banach space~$\fCONTROL_{\follower}$.
Consider~$\Game$ the class containing~$\game = (\Cons_{\follower},\Pref_{\follower})_{\follower \in \Follower}$, where~$\Cons_{\follower} : \fControl_{-\follower} \multimap \fControl_{\follower}$ and~$\Pref_{\follower} : \fControl \multimap \fControl_{\follower}$, satisfying the following two conditions
\begin{enumerate}[label=($\Game_{\arabic*}$), leftmargin=*]
  \item\label{asmp:A1_to_A3} the assumptions~\ref{asmp:A1}--\ref{asmp:A3} hold,
  \item\label{asmp:A4_strengthened} there exists~$\tau > 0$ such that for all~$\follower \in \Follower$,
    \begin{align*}
      \fcontrol_{\follower} \in \Cons_{\follower}(\fcontrol_{-\follower}) \implies \Nhood_{\tau}[\fcontrol_{\follower}] \cap \Pref_{\follower}(\fcontrol) = \emptyset.
    \end{align*}
\end{enumerate}
One could directly observe that~\ref{asmp:A4_strengthened} implies~\ref{asmp:A4}.
Therefore any~$\game = (\Cons_{\follower},\Pref_{\follower})_{\follower \in \Follower} \in \Game$ induces a nonempty solution set~$\NE(\game)$, where
\begin{align*}
  \NE(\game) := \set{\fcontrol \in \fControl \mid \fcontrol_{\follower} \in \Cons_{\follower}(\fcontrol_{-\follower}), \quad \Cons_{\follower}(\fcontrol_{-\follower}) \cap \Pref_{\follower}(\fcontrol) = \emptyset, \quad \forall \follower \in \Follower}.
\end{align*}

Given~$\IMPROVEMENT := (\IMPROVEMENT_{\follower})_{\follower \in \Follower}$, where for each~$\follower \in \Follower$,~$\IMPROVEMENT_{\follower} \subset \fControl$ is open relative to~$\fControl$, we define a subclass~$\Game(\IMPROVEMENT) \subset \Game$ by
\begin{align*}
  \Game(\IMPROVEMENT) := \set{\game \in \Game \mid \Improvement_{\follower}(\game) = \IMPROVEMENT_{\follower}, \quad \forall \follower \in \Follower},
\end{align*}
where~$\Improvement_{\follower}(\game) := \set{\fcontrol \in \fControl \mid \Cons_{\follower} (\fcontrol_{-\follower}) \cap \Pref_{\follower}(\fcontrol) \neq \emptyset}$.
On this subset~$\Game(\IMPROVEMENT)$, we impose a distance function~$\rho : \Game(\IMPROVEMENT) \times \Game(\IMPROVEMENT) \to \R_{+}$ given by
\begin{align*}
  \rho(\game,\hat{\game}) = 
  \sum_{\follower \in \Follower} \left[ \sup_{\fcontrol \in \fControl} h(\Cons_{\follower}(\fcontrol_{-\follower}), \hat{\Cons}_{\follower}(\fcontrol_{-\follower})) + \sup_{\fcontrol \in \fControl} h(\Pref_{\follower}(\fcontrol), \hat{\Pref}_{\follower}(\fcontrol)) \right],
\end{align*}
where~$\game = (\Cons_{\follower}, \Pref_{\follower})_{\follower \in \Follower}$ and~$\hat{\game} = (\hat{\Cons}_{\follower}, \hat{\Pref}_{\follower})_{\follower \in \Follower}$.
The pair~$(\Game(\IMPROVEMENT),\rho)$ is clearly a metric space, but its completeness is not so easy to see.
In fact, to prove its completeness, we need to assume a regularity condition below.
\begin{definition}
  A closed subset~$\regGame \subset \Game(\IMPROVEMENT)$ is said to be \emph{regular w.r.t.~$\IMPROVEMENT$} if
  \begin{enumerate}[label=$(\regGame_{\arabic*})$, leftmargin=*]
    \item\label{asmp:regGame1} the condition~\ref{asmp:A4_strengthened} holds with the same~$\tau > 0$ for any~$\game \in \regGame$,
    \item\label{asm:regGame2} there exists~$\alpha > 0$ such that for any~$\varepsilon > 0$, any~$\game = (\Cons_{\follower},\Pref_{\follower})_{\follower \in \Follower} \in \regGame$, and any~$\fcontrol \in \Improvement_{\follower} = \IMPROVEMENT_{\follower}$, it holds
  \begin{align*}
    \Nhood_{\varepsilon}[\Cons_{\follower}(\fcontrol_{-\follower})] \cap \Nhood_{\varepsilon}[\Pref_{\follower}(\fcontrol)] \subset \Nhood_{\alpha\varepsilon}[\Cons_{\follower}(\fcontrol_{-\follower}) \cap \Pref_{\follower}(\fcontrol)].
  \end{align*}
  \end{enumerate}
\end{definition}

Following directly from the defintion, one could observe that the above regularity always holds in finite dimensional spaces when the intersections~$\Cons_{\follower}(\fcontrol_{-\follower}) \cap \Pref_{\follower}(\fcontrol)$ are bounded.
\begin{lemma}\label{lem:completeness}
  If~$\regGame \subset \Game(\IMPROVEMENT)$ is regular w.r.t.~$\IMPROVEMENT$, then~$(\regGame,\rho)$ is complete.
\end{lemma}
\begin{proof}
  Let~$(\game^{n}) = (\Cons_{\follower}^{n},\Pref_{\follower}^{n})_{\follower \in \Follower}^{n \in \N}$ be a Cauchy sequence in~$\regGame$.
  For every~$\fcontrol \in \fControl$ and~$\follower \in \Follower$, the sequence~$(\Cons_{\follower}^{n}(\fcontrol_{-\follower}))^{n \in \N}$ of compact sets is Cauchy w.r.t. to the Hausdorff metric on~$\Compact(\fControl_{\follower})$, which in turn converges, in the Hausdorff metric, to a compact set~$\hat{\Cons}_{\follower}(\fcontrol_{-\follower}) \in \Compact(\fControl_{\follower})$.
  In the same way, the sequence~$(\Pref_{\follower}^{n}(\fcontrol))^{n \in \N}$ converges to a set~$\hat{\Pref}_{\follower}(\fcontrol) \in \Compact(\fControl_{\follower})$.
  Through these limits, we have defined the maps~$\hat{\Cons}_{\follower} : \fControl_{-\follower} \multimap \fControl_{\follower}$ and~$\hat{\Pref}_{\follower} : \fControl \multimap \fControl_{\follower}$ which produce compact convex values.
  The goal is to prove that~$\hat{\game} := (\hat{\Cons}_{\follower}, \hat{\Pref}_{\follower})_{\follower \in \Follower}$ belongs to~$\Game$, through the three claims below.

  {\bfseries Claim 1:} The maps~$\hat{\Cons}_{\follower}$ and~$\hat{\Pref}_{\follower}$ are upper semicontinuous.

  Take any sequences~$(\fcontrol^{n})$ and~$(\fycontrol^{n})$ in~$\fControl$ that are convergent, respectively, to~$\hat{\fcontrol} \in \fControl$ and~$\hat{\fycontrol} \in \fControl$ with the property that~$\fycontrol^{n} \in \hat{\Cons}_{\follower}(\fcontrol^{n})$ for all~$n \in \N$.
  It suffices to show that~$\hat{\fycontrol} \in \hat{\Cons}_{\follower}(\hat{\fcontrol})$.
  Now, for any~$\follower \in \Follower$ and $r,n \in \N$, we get
  \begin{align}
    d(\fycontrol^{n}, \hat{\Cons}_{\follower}(\hat{\fcontrol}))
    &\leq
    d(\fycontrol^{n}, \hat{\Cons}_{\follower}(\fcontrol^{n})) + h(\hat{\Cons}_{\follower}(\fcontrol^{n}), \Cons_{\follower}^{r}(\fcontrol^{n})) + e(\Cons_{\follower}^{r}(\fcontrol^{n}), \hat{\Cons}_{\follower}(\hat{\follower})) \label{eqn:triangle1} \\
    &=
    h(\hat{\Cons}_{\follower}(\fcontrol^{n}), \Cons_{\follower}^{r}(\fcontrol^{n})) + e(\Cons_{\follower}^{r}(\fcontrol^{n}), \hat{\Cons}_{\follower}(\hat{\fcontrol})). \label{eqn:triangle2}
  \end{align}
  Let~$\varepsilon > 0$ be arbitrary.
  Since~$\lim_{r \to \infty} h(\Cons_{\follower}^{r}(\fcontrol), \hat{\Cons}_{\follower}(\fcontrol))$ uniformly on~$\fcontrol \in \fControl$, there exists~$R > 0$ in which 
  \begin{align}
    h(\hat{\Cons}_{\follower}(\fcontrol), \Cons_{\follower}^{r}(\fcontrol)) < \frac{\varepsilon}{2} \label{eqn:eps/2--1}
  \end{align}
  for any~$r > R$ and~$\fcontrol \in \fControl$.
  Take any~$r > R$, the above inequality gives
  \begin{align*}
    \Cons_{\follower}^{r}(\hat{\fcontrol}) \subset \Nhood_{\varepsilon/\alpha}[\hat{\Cons}_{\follower}(\hat{\fcontrol})].
  \end{align*}
  Together with the facts that each map~$\Cons_{\follower}^{r}$ is upper semicontinuous and~$\fcontrol^{n} \to \hat{\fcontrol}$, there exists~$N > 0$ in which
  \begin{align*}
    \Cons_{\follower}^{r}(\fcontrol^{n}) \subset \Nhood_{\varepsilon/2}[\hat{\Cons}_{\follower}(\hat{\fcontrol})]
  \end{align*}
  for all~$n > N$.
  This means, for any~$r,n > \max\set{R,N}$, it holds that
  \begin{align}
    e(\Cons_{\follower}^{r}(\fcontrol^{n}), \hat{\Cons}_{\follower}(\hat{\fcontrol})) < \frac{\varepsilon}{2}. \label{eqn:eqn:eps/2--2}
  \end{align}
  Combining~\eqref{eqn:triangle1}--\eqref{eqn:triangle2} with~\eqref{eqn:eps/2--1} and~\eqref{eqn:eqn:eps/2--2}, we obtain that
  \begin{align*}
    d(\fycontrol^{n}, \hat{\Cons}_{\follower}(\hat{\fcontrol})) < \varepsilon
  \end{align*}
  for each~$n > \max\set{R,N}$.
  It implies that~$\hat{\fycontrol} \in \hat{\Cons}_{\follower}(\hat{\fcontrol})$, and we conclude that~$\hat{\Cons}_{\follower}$ is upper semicontinuous.

  With a similar technique, we could show that~$\hat{\Pref}_{\follower}$ is upper semicontinuous, so that Claim 1 is proved.

  {\bfseries Claim 2:} $\Improvement_{\follower}(\hat{\game}) = \IMPROVEMENT_{\follower}$ for all~$\follower \in \Follower$.

  This claim follows directly from the regularity of~$\regGame$ w.r.t~$\IMPROVEMENT$.

  {\bfseries Claim 3:} For $\fcontrol \in \fControl$,~$\fcontrol_{\follower} \in \hat{\Cons}_{\follower}(\fcontrol_{-\follower})$ implies~$\Nhood_{\tau}[\fcontrol_{\follower}] \cap \hat{\Pref}_{\follower}(\fcontrol) = \emptyset$ for all~$\fcontrol \in \fControl$.

  Let~$\follower \in \Follower$ and take any~$\hat{\fcontrol} \in \fControl$ with~$\hat{\fcontrol}_{\follower} \in \hat{\Cons}_{\follower}(\hat{\fcontrol}_{-\follower})$.
  Using the PK convergence of~$(\Cons_{\follower}^{n})$ to its limit~$\hat{\Cons}_{\follower}$, there exists a sequence~$(\fcontrol^{n})$ in~$X$ such that~$\fcontrol^{n} \to \hat{\fcontrol}$ and~$\fcontrol_{\follower}^{n} \in \Cons_{\follower}^{n}(\hat{\fcontrol}_{-\follower})$ for all~$n \in \N$.
  Applying~\ref{asmp:regGame1} to the point~$(\fcontrol_{\follower}^{n},\hat{\fcontrol}_{-\follower})$ with~$\Cons_{\follower}^{n}(\hat{\fcontrol}_{-\follower})$, we obtain
  \begin{align*}
    \tau \leq d(\fcontrol_{\follower}^{n}, \Pref_{\follower}^{n}(\hat{\fcontrol})) \leq d(\fcontrol_{\follower}^{n}, \hat{\Pref}_{\follower}(\hat{\fcontrol})) + h(\Pref_{\follower}^{n}(\hat{\fcontrol}), \Pref_{\follower}(\hat{\fcontrol}))
  \end{align*}
  for large~$n \in \N$.
  Letting~$n \to \infty$, we have~$\tau \leq d(\hat{\fcontrol}_{\follower}, \hat{\Pref}_{\follower}(\hat{\fcontrol}))$, which is equivalent to saying~$\Nhood_{\tau}[\hat{\fcontrol}_{\follower}] \cap \Pref_{\follower}(\hat{\fcontrol}) = \emptyset$
  The claim is thus proved.

  Accumulating the Claims 1--3, one conclude that~\ref{asmp:A1_to_A3} and~\ref{asmp:A4_strengthened} are satisfied at the limit~$\hat{\game}$.
  This proves~$\hat{\game} \in \Game$ and the closedness of~$\regGame$ guarantees that it is complete.
\end{proof}

Following the proof of Theorem~\ref{thm:existence}, one have a fixed point characterization of elements in~$\NE(\cdot)$.
We define a function~$\gapfun : \Game \times \fControl \to \R_{+}$ by
\begin{align*}
  \gapfun(\game;\fcontrol) &:= \max_{\follower \in \Follower} d(\fcontrol_{\follower},\Gmap_{\follower}(\game;\fcontrol)), \\
  \text{where}\quad \Gmap_{\follower}(\game;\fcontrol) &:= \begin{cases}
    \Cons_{\follower}(\fcontrol_{-\follower}) \cap \Pref_{\follower}(\fcontrol) &\text{if~$\fcontrol \in \Improvement_{\follower}(\game)$}, \\
    \Cons_{\follower}(\fcontrol_{-\follower}) &\text{otherwise},
  \end{cases}
\end{align*}
for any~$\game = (\Cons_{\follower},\Pref_{\follower})_{\follower \in \Follower} \in \Game$.
One may observe that~$\gapfun(\game;\cdot)$ is a gap function for the game~$\game$, {\itshape i.e.} it is nonnegative and~$\gapfun(\game;\fcontrol) = 0$ if and only if~$\fcontrol \in \NE(\game)$.
\begin{lemma}\label{lem:lsc}
  The gap function~$\gapfun : \Game \times \fControl \to \R_{+}$ is lower semicontinuous on~$\regGame \times \fControl$, provided that~$\regGame \subset \Game(\IMPROVEMENT)$ is regular w.r.t.~$\IMPROVEMENT$.
\end{lemma}
\begin{proof}
  Let~$(\game^{n})$,~$\game^{n} = (\Cons_{\follower}^{n},\Pref_{\follower}^{n})_{\follower \in \Follower}$, be a sequence in~$\regGame$ that is convergent to~$\hat{\game} = (\hat{\Cons}_{\follower},\hat{\Pref}_{\follower}) \in \regGame$.

  {\bfseries Claim 1:} For any~$\follower \in \Follower$,~$\Gmap_{\follower}(\game^{n};\fcontrol) \to \Gmap_{\follower}(\hat{\game};\fcontrol)$ uniformly over~$\fcontrol \in \fControl$.

  It is sufficient to show that~$\Cons_{\follower}^{n}(\fcontrol_{-\follower}) \cap \Pref_{\follower}^{n}(\fcontrol) \to \hat{\Cons}_{\follower}(\fcontrol_{-\follower}) \cap \hat{\Pref}_{\follower}(\fcontrol)$ uniformly on~$\fcontrol \in \Improvement_{\follower}$.
  It is clear that we have
  \begin{align}\label{eqn:uniform_convergence_Cons_Pref}
    \Cons_{\follower}^{n}(\fcontrol) \to \hat{\Cons}_{\follower}(\fcontrol) \quad \text{and} \quad \Pref_{\follower}^{n}(\fcontrol) \to \hat{\Pref}_{\follower}(\fcontrol)
  \end{align}
  uniformly on~$\fcontrol \in \fControl$.
  Let us fix~$\follower \in \Follower$ and~$\fcontrol \in \Improvement_{\follower}$.
  For any~$\varepsilon > 0$, and~$n \in \N$ sufficiently large, we have
  \begin{alignat*}{2}
    & \hat{\Cons}_{\follower}(\fcontrol_{-\follower}) \subset \Nhood_{\varepsilon/\alpha}[\Cons_{\follower}^{n}(\fcontrol_{-\follower})], \quad
    && \Cons_{\follower}^{n}(\fcontrol_{-\follower}) \subset \Nhood_{\varepsilon/\alpha}[\hat{\Cons}_{\follower}(\fcontrol_{-\follower})] \\
    & \hat{\Pref}_{\follower}(\fcontrol) \subset \Nhood_{\varepsilon/\alpha}[\Pref_{\follower}^{n}(\fcontrol)], \quad
    && \Pref_{\follower}^{n}(\fcontrol) \subset \Nhood_{\varepsilon/\alpha}[\hat{\Pref}_{\follower}(\fcontrol)].
  \end{alignat*}
  For large~$n \in \N$, the regularity of~$\regGame$ yields
  \begin{alignat*}{3}
    \hat{\Cons}_{\follower}(\fcontrol_{-\follower}) \cap \hat{\Pref}_{\follower}(\fcontrol) 
    &\subset \Nhood_{\varepsilon/\alpha}[\Cons_{\follower}^{n}(\fcontrol_{-\follower})] \cap \Nhood_{\varepsilon/\alpha}[\Pref_{\follower}^{n}(\fcontrol)] 
    &&\subset \Nhood_{\varepsilon}[\Cons_{\follower}^{n}(\fcontrol_{-\follower}) \cap \Pref_{\follower}^{n}(\fcontrol)] \\
    \Cons_{\follower}^{n}(\fcontrol_{-\follower}) \cap \Pref_{\follower}^{n}(\fcontrol) 
    &\subset \Nhood_{\varepsilon/\alpha}[\hat{\Cons}_{\follower}(\fcontrol_{-\follower})] \cap \Nhood_{\varepsilon/\alpha}[\hat{\Pref}_{\follower}(\fcontrol)] 
    &&\subset \Nhood_{\varepsilon}[\hat{\Cons}_{\follower}(\fcontrol_{-\follower}) \cap \hat{\Pref}_{\follower}(\fcontrol)].
  \end{alignat*}
  This means we have
  \begin{align*}
    h(\hat{\Cons}_{\follower}(\fcontrol_{-\follower}) \cap \hat{\Pref}_{\follower}(\fcontrol), \Cons_{\follower}^{n}(\fcontrol_{-\follower}) \cap \Pref_{\follower}^{n}(\fcontrol)) < \varepsilon
  \end{align*}
  whenever~$n \in \N$ is sufficiently large.
  Since this is uniform over~$\fcontrol \in \Improvement_{\follower}$, we conclude that
  \begin{align*}
    \sup_{\fcontrol \in \Improvement_{\follower}} h(\hat{\Cons}_{\follower}(\fcontrol_{-\follower}) \cap \hat{\Pref}_{\follower}(\fcontrol), \Cons_{\follower}^{n}(\fcontrol_{-\follower}) \cap \Pref_{\follower}^{n}(\fcontrol)) \to 0.
  \end{align*}
  Together with the uniform convergence of~$(\Cons_{\follower}^{n}(\fcontrol))$ over~$\fcontrol \in \fControl \setminus \Improvement_{\follower}$, which could be observed in~\eqref{eqn:uniform_convergence_Cons_Pref}.
  This proves the claim.

  Now, note that~$\Gmap_{\follower}(\game^{n};\cdot)$ and~$\Gmap_{\follower}(\hat{\game};\cdot)$ are upper semicontinuous, and hence the functions~$\fcontrol \mapsto d(\fcontrol,\Gmap_{\follower}(\game^{n};x))$ and~$\fcontrol \mapsto d(\fcontrol,\Gmap_{\follower}(\hat{\game};\cdot))$ are both lower semicontinuous (by~\cite[Corollary 1.4.17]{zbMATH00045282}).
  Take any convergent sequence~$(\fcontrol^{n})$ in~$\fControl$ with the limit~$\hat{\fcontrol} \in \fControl$.
  Let~$\varepsilon > 0$.
  For~$n \in \N$ sufficiently large, we have
  \begin{align*}
    d(\fcontrol^{n}, \Gmap_{\follower}(\hat{\game};\fcontrol^{n})) > d(\hat{\fcontrol}, \Gmap_{\follower}(\hat{\game};\hat{\fcontrol})) - \frac{\varepsilon}{2}.
  \end{align*}
  Using the compactness, for any~$n \in \N$, there exists~$\fucontrol^{n} \in \Gmap_{\follower}(\game^{n};\fcontrol^{n})$ such that
  \begin{align*}
    d(\fcontrol^{n},\fucontrol^{n}) = d(\fcontrol^{n}, \Gmap_{\follower}(\game^{n};\fcontrol^{n})).
  \end{align*}
  From the convergence proved in Claim 1, for large enough~$n \in \N$, we could pick~$\fvcontrol^{n} \in \Gmap_{\follower}({\hat{\game};\fcontrol^{n}})$ such that
  \begin{align*}
    d(\fucontrol^{n},\fvcontrol^{n}) = d(\fucontrol^{n}, \Gmap_{\follower}(\hat{\game};\fcontrol^{n})) \leq h(\Gmap_{\follower}(\game^{n};\fcontrol^{n}), \Gmap_{\follower}(\hat{\game};\fcontrol^{n})) < \frac{\varepsilon}{2}.
  \end{align*}
  Therefore, we have
  \begin{align*}
    d(\fcontrol^{n}, \Gmap_{\follower}(\game^{n};\fcontrol^{n}))
    &=
    d(\fcontrol^{n},\fucontrol^{n}) \\
    &\geq d(\fcontrol^{n}, \fvcontrol^{n}) - d(\fvcontrol^{n},\fucontrol^{n}) \\
    &\geq d(\fcontrol^{n}, \Gmap_{\follower}(\hat{\game};\fcontrol^{n})) - d(\fvcontrol^{n},\fucontrol^{n}) \\
    &> d(\hat{\fcontrol},\Gmap_{\follower}(\hat{\game};\hat{\fcontrol})) - \varepsilon
  \end{align*}
  for large~$n \in \N$.
  Repeat the process for all~$\follower \in \Follower$, we get
  \begin{align*}
    \gapfun(\game^{n};\fcontrol^{n}) > \gapfun(\hat{\game};\hat{\fcontrol}) - \varepsilon
  \end{align*}
  for large~$n \in \N$.
  Thus, the function~$\gapfun$ is lower semicontinuous on~$\regGame \times \fControl$.
\end{proof}

We finally have all the tools to prove our main stability theorem.
We now state and prove the continuity of~$\NE(\cdot)$ based on Lemmas~\ref{lem:completeness} and~\ref{lem:lsc}.
\begin{theorem}\label{thm:stability}
  The map~$\game \mapsto \NE(\game)$ is upper semicontinuous on~$\Psi \subset \Game(\IMPROVEMENT)$, provided that~$\Psi$ is regular w.r.t.~$\IMPROVEMENT$.
\end{theorem}
\begin{proof}
  Observe that, for each~$\game = (\Cons_{\follower},\Pref_{\follower})_{\follower \in \Follower} \in \regGame$, one may now deduce that
  \begin{align*}
    \NE(\game) = \S(\game) := \set{\fcontrol \in \fControl \mid \gapfun(\game;\fcontrol) = 0} = \argmin_{\fcontrol \in \fControl} \gapfun(\game;\fcontrol).
  \end{align*}
  This set above is nonempty for all~$\game \in \regGame$ in view of Theorem~\ref{thm:existence}.
  To show that~$\NE = \S$ is upper semicontinuous, we take a convergent sequence~$(\game^{n})$ in~$\regGame$ with the limit~$\hat{\game} \in \regGame$.
  Let~$(\fcontrol^{n})$ be a sequence in~$\fControl$ such that~$\fcontrol^{n} \in \S(\game^{n})$, for all~$n \in \N$, and is convergent to~$\hat{\fcontrol} \in \fControl$.
  This means~$\gapfun(\game^{n};\fcontrol^{n}) = 0$ for all~$n \in \N$.
  The lower semicontinuity of~$\gapfun$ implies that~$\gapfun(\hat{\game};\hat{\fcontrol}) = 0$ and that~$\hat{\fcontrol} \in \S(\hat{\game})$.
  Therefore, the map~$\NE = \S$ is upper semicontinuous on~$\regGame$.
\end{proof}

\section{An existence result for SLMFG with direct preference maps}

This final section presents an existence result for the SLMFG with direct preference maps of the form~\eqref{eqn:SLMFG}, with a more specific structure of the followers' problem.
The theorem is, in fact, just an application of the Weierstra\ss{} theorem combined with the stability result from Theorem~\ref{thm:stability}.

\begin{theorem}\label{thm:SLMFG_existence}
  Suppose that
  \begin{enumerate}[label=\upshape(\alph*), leftmargin=*]
    \item\label{asmp:exist1} $\lControl$ is a topological space and~$\lCons \subset \lControl$ is closed,
    \item\label{asmp:exist2} $\lcriterion : \lControl \times \fControl \to \R$ is lower semicontinuous,
    \item\label{asmp:exist3} $\signal : \lControl \to \regGame$ is continuous, where~$\regGame \subset \Game(\IMPROVEMENT)$ is regular w.r.t.~$\IMPROVEMENT$.
  \end{enumerate}
  Then the following SLMFG has a solution:
  \begin{subequations}\label{eqn:SLMFG_thm}
    \begin{empheq}[left=\empheqlbrace]{align}
      \min_{\lcontrol,\fcontrol} \quad
        & \lcriterion(\lcontrol,\fcontrol) \\
        \text{s.t.}  \quad
        & \lcontrol \in \lCons  \\
        & \signal(\lcontrol) = (\Cons_{\follower}^{\lcontrol},\Pref_{\follower}^{\lcontrol})_{\follower \in \Follower} \label{eqn:signal}\\
        & \forall \follower \in \Follower: \label{eqn:ne1}\\
        & \ \lfloor \ \fcontrol_{\follower} \in \Cons_{\follower}^{\lcontrol}(\fcontrol_{-\follower}), \quad \Cons_{\follower}^{\lcontrol}(\fcontrol_{-\follower}) \cap \Pref_{\follower}^{\lcontrol}(\fcontrol) = \emptyset.  \label{eqn:ne2}
    \end{empheq}
  \end{subequations}
\end{theorem}
\begin{proof}
  One could observe that the constraints~\eqref{eqn:signal}--\eqref{eqn:ne2} reduce simply to~$\fcontrol \in \NE(\psi(\lcontrol))$, or equivalently, to~$(\lcontrol,\fcontrol) \in \gr(\NE \circ \signal)$.
  This means the SLMFG~\eqref{eqn:SLMFG_thm} reduces to
  \begin{subequations}\label{eqn:SLMFG_thm_reformed}
    \begin{empheq}[left=\empheqlbrace]{align}
      \min_{\lcontrol,\fcontrol} \quad
        & \lcriterion(\lcontrol,\fcontrol) \\
        \text{s.t.}  \quad
        & \lcontrol \in \lCons  \\
        & (\lcontrol,\fcontrol) \in \gr(\NE \circ \signal).
    \end{empheq}
  \end{subequations}
  From Theorem~\ref{thm:stability}, the response map~$\NE$ is upper semicontinuous.
  Since~$\signal$ is continuous, the composition~$\NE \circ \signal : \lControl \multimap \fControl$ is again upper semicontinuous.
  As a consequent, the graph of~$\NE \circ \signal$ is closed due to the compactness of~$\fControl$.
  The Weierstra\ss{} theorem then guarantees that~\eqref{eqn:SLMFG_thm_reformed}, and hence~\eqref{eqn:SLMFG_thm}, has a solution.
\end{proof}

In the formulation of SLMFG~\eqref{eqn:SLMFG_thm}, the leader's action~$\lcontrol$ is passed to the followers through the \emph{signal map}~$\psi$, which explicitly selects an abstract economy profile~$\game \in \regGame$.
One could also present Theorem~\ref{thm:SLMFG_existence} (and the problem~\eqref{eqn:SLMFG_thm} herein) more transparently by adopting the parametric approach, {\itshape i.e.,} each follower~$\follower \in \Follower$ has the parametrized constraint map $\Cons_{\follower} : \lControl \times \fControl_{-\follower} \multimap \fControl_{\follower}$ and the parametrized preference map~$\Pref_{\follower} : \lControl \times \fControl \multimap \fControl_{\follower}$.
An alternative presentation of Theorem~\ref{thm:SLMFG_existence} now reads:
If the conditions~\ref{asmp:exist1} and~\ref{asmp:exist2} hold, and the tuple~$(\Cons_{\follower}(\lcontrol,\cdot), \Pref_{\follower}(\lcontrol,\cdot))_{\follower \in \Follower}$ belongs to~$\regGame$, then the corresponding SLMFG of the form~\eqref{eqn:SLMFG} has a solution.



\section*{Acknowledgments}

The authors are grateful to the peer reviewers for their valuable comments, which have had a great impact on the presentation of this paper.
The first author was supported by the Petchra Pra Jom Klao M.Sc. Research Scholarship from King Mongkut's University of Technology Thonburi (Grant No.19/2564).
This research project is supported by Thailand Science Research and Innovation (TSRI) Basic Research Fund under project number FRB670073/0164.

\section*{Author Contributions}

Supervision: Poom Kumam, Parin Chaipunya; 
Investigation: Yutthakan Chummongkhon;
Writing -- Original draft: Yutthakan Chummongkhon, Parin Chaipunya; 
Writing -- Review \& editing, Poom Kumam, Parin Chaipunya.

\renewcommand\bibname{References}


\end{document}